\documentclass[12pt,reqno]{amsart}
 \usepackage[numbers,sort&compress]{natbib}
\usepackage{amsfonts}
\usepackage{amssymb,color}
\usepackage{amssymb}

\usepackage[colorlinks=true]{hyperref}
\hypersetup{urlcolor=blue, citecolor=blue}

\usepackage[margin=3cm, a4paper]{geometry}

\def\normo#1{\left\|#1\right\|}
\def\normb#1{\big\|#1\big\|}

\def\brk#1{\left(#1\right)}

\def\norm#1{\|#1\|}

\renewcommand{\sigma}{\omega}

\newcommand{\T}{{\mathbb T}}
\newcommand{\E}{{\mathbb K}}
\newcommand{\R}{{\mathbb R}}

\newcommand{\Z}{{\mathbb Z}}

\def\norm#1{\|#1\|}
\def\normo#1{\left\|#1\right\|}
\def\normb#1{\big\|#1\big\|}

\newcommand{\e}{\varepsilon}

\newcommand{\EQ}[1]{\begin{align*}\begin{split} #1 \end{split}\end{align*}}
\newcommand{\EQn}[1]{\begin{align}\begin{split} #1 \end{split}\end{align}}
\newcommand{\Ncase}[1]{\begin{equation} \left\{\begin{aligned}#1\end{aligned}\right. \end{equation}}
\newcommand{\CAS}[1]{\begin{cases} #1 \end{cases}}

\newcommand{\Del}[1]{}

\numberwithin{equation}{section}

\newtheorem{thm}{Theorem}[section]

\newtheorem{lem}[thm]{Lemma}

\theoremstyle{remark}
\newtheorem{rem}[thm]{Remark}

\begin{document}
\subjclass[2010]{Primary 35G25, 37K10. Secondary 35Q53}
\keywords{Camassa-Holm equation, shallow water wave models, norm inflation, ill-posedness,
critical space}

\title[Camassa-Holm equation]{
Ill-posedness of the Camassa-Holm and related equations in the critical space}

\author{Zihua Guo}
\address{(Z. Guo) School of Mathematical Sciences, Monash University, Melbourne, VIC 3800, Australia}
\email{zihua.guo@monash.edu}
\thanks{}

\author{Xingxing Liu}
\address{(X. Liu)  School of Mathematics, China University of Mining and Technology, Xuzhou, 221116, China}
\email{liuxxmaths@cumt.edu.cn}
\thanks{}

\author{Luc Molinet}
\address{(L. Molinet) Institut Denis Poisson  (CNRS UMR 7013),  Universit\'{e} de Tours, Universit\'{e} d'Orl\'eans, Parc Grandmont, 37200 Tours, France}
\email{luc.molinet@univ-tours.fr}
\thanks{}

\author{Zhaoyang Yin}
\address{(Z. Yin) School of Mathematics, Sun Yat-sen University, Guangzhou, 510275, China}
\email{mcsyzy@mail.susy.edu.cn}
\thanks{}

\begin{abstract}
We prove norm inflation and hence ill-posedness for a class of shallow water wave equations, such as the Camassa-Holm equation, Degasperis-Procesi equation and Novikov equation etc., in the critical Sobolev space $H^{3/2}$ and even in the Besov space $B^{1+1/p}_{p,r}$ for $p\in [1,\infty], r\in (1,\infty]$. Our results cover both real-line and torus cases (only real-line case for Novikov), solving an open problem left in the previous works (\cite{Danchin2,Byers,HHK}).
\end{abstract}

\maketitle


\section{Introduction}
The aim of this paper is to address the question of the well-posedness of the Cauchy problems to a class of shallow water wave equations, such as the Camassa-Holm equation, Degasperis-Procesi equation, Novikov equation and other related models, in the Sobolev space $H^{3/2}$ or Besov spaces $B_{2,r}^{3/2},$ $1<r<\infty$, and in both real-line and torus cases (only real-line case for Novikov). The methods we used in this paper are very simple and apply equally well to these  equations. Thus, we focus on Camassa-Holm equation and give the results for other equations as remarks.

Consider the Cauchy problem to the Camassa-Holm (CH) equation
\Ncase{\label{eq:CH}
u_t-u_{txx}+3uu_x &= 2u_xu_{xx}+uu_{xxx},\quad    t>0,\, x\in \E,\\
\hfill{} u(0,x)  &= u_0(x),
}
with $\E=\R$ or the torus $\E=\T=\R/2\pi \Z$, which models the unidirectional propagation of shallow water waves over a flat bottom. Here the function $u(t,x)$ represents the water’s free surface above a flat bottom. The CH equation \eqref{eq:CH} appeared initially in the context of hereditary symmetries studied by Fuchssteiner and Fokas \cite{F-Fokas} as a bi-Hamiltonian generalization of KdV equation. Later, Camassa and Holm \cite{C-H} derived it by approximating directly in the Hamiltonian for Euler's equations in the shallow water regime.

After the CH equation \eqref{eq:CH} was derived physically in the context of water waves, there are a large amount of literatures devoted to its study and which we do not attempt to exhaust. Here we only recall some results concerning the well-posedness of the Cauchy problem \eqref{eq:CH} (see also the survey \cite{Molinet}). Li and Olver \citep{Li-Olver} (see also \cite{Rodriguez-Blanco}) proved that the Cauchy problem \eqref{eq:CH} is locally well-posed with
the initial data $u_0(x)\in H^s(\R)$ with $s>3/2$ (See  \cite{C-E1} for earlier results in $H^s(\R)$, $s\geq 3$). The analogous well-posedness result on the torus was shown in \cite{Misiolek1} (See \cite{C-E2} for earlier result
in $H^3(\T)$). Danchin \cite{Danchin1} considered the local well-posedness in the Besov space, and proved well-posedness in $B^s_{p,r}$ if $1\leq p\leq \infty $,  $1\leq r< \infty$ and $s>\max\{1+1/p,3/2\}$ (For continuous dependence, see Li and Yin \cite{Li-Yin}). Note that if one wants to include the case $ r=\infty $ then one has to weaken the notion of well-posedness since the continuity of the solution with values in $B^s_{p,\infty}$ as well as the continuity of the data-to-solution map with values in $ L^\infty(0,T; B^s_{p,\infty} )$ are not known to hold.
For the endpoints, local well-posedness in the space $B^{3/2}_{2,1}$, and ill-posedness in $B^{3/2}_{2,\infty}$ (the data-to-solution map $u_0\mapsto u$ is not continuous using peakon solution. Note however that, as mentioned above, the continuity of this map is not known for any $s\in\R $) were established in \cite{Danchin2}.
Moreover, as mentioned in \cite{Danchin2}, there is no definitive answer to the intermediate cases $B^{3/2}_{2,r},1<r<\infty$.  The critical Sobolev space for well-posedness is $H^{3/2}$ since CH is ill-posed  in $H^s$ for $s<3/2$ in the sense of norm inflation (Byers \cite{Byers}).  For other equations there are similar results, but we do not list here.  To the best of our knowledge, well-posedness in $H^{3/2}$ is still an open problem.

In this paper we solve this problem by proving
the norm inflation and hence the ill-posedness of the CH equation \eqref{eq:CH} in $H^{3/2}$ and  in $B^{3/2}_{2,r},\ 1<r<\infty$.

\begin{thm}\label{thm1}
Let $\E=\R\mbox{ or }\T$, $1\leq p\leq \infty$ and $1<r\leq \infty$. $\forall\ \e>0$, there exists $u_0 \in  H^\infty(\E)$, real-valued, such that the following hold:

(1) $\norm{u_0}_{B^{1+1/p}_{p,r}(\E)}\leq \e$;

(2) There is a unique solution $u\in C([0,T); H^\infty(\E))$ to the Cauchy problem of \eqref{eq:CH} with a maximal lifespan $T<\e$;

(3) $\limsup_{t\to T^-}\norm{u(t)}_{B^{1+1/p}_{p,r}(\E)}\geq \limsup_{t\to T^-}\norm{u(t)}_{B^{1}_{\infty,\infty}(\E)}=\infty$.
\end{thm}

\begin{rem}\label{remmain}
We would like to give a few remarks.

(1) Taking $p=r=2$, we see that our theorem implies the ill-posedness in the critical Sobolev space $H^{3/2}$ (non-continuity of the flow map at 0). The norm inflation is even stronger than the classical sense, namely weak norm $\norm{u(t)}_{B^{1}_{\infty,\infty}(\E)}$ blows up.

Recently, for Euler equation, Bourgain and Li in \cite{BL, BL1} proved strong local ill-posedness in borderline Besov spaces $B^{d/p+1}_{p,r}$ for $1\leq p<\infty$ and $1<r\leq\infty$ when $d=2,3$.  The hyperbolic structure of CH and other equations is similar to Euler equation.  Compared to their proof, our proof of Theorem \ref{thm1} is much simpler since for CH we have simple blow-up result using symmetry. Our methods are inspired by the works of Saut on Burgers equation (see \cite{Saut, Saut2}).

(2) Theorem \ref{thm1} also holds for the following related equations:
\begin{itemize}
\item Degasperis-Procesi (DP) equation
\begin{align}\label{eq:DP}
u_t-u_{txx}+4uu_x=3u_xu_{xx}+uu_{xxx}, \quad t>0, \, x\in \E.
\end{align}

\item General $b$-family of equations with $1<b\leq 3$:
\begin{align}\label{eq:b-family}
u_t-u_{txx}+(b+1)uu_x=bu_xu_{xx}+uu_{xxx}, \quad t>0,\, x\in \E.
\end{align}
Note that CH corresponds to $b=2$ and DP corresponds to $b=3$.
\end{itemize}

(3) Theorem \ref{thm1} also holds for the Novikov equation on the real-line (See Remark \ref{remlast}):
\EQn{\label{eq:Novikov}
u_t-u_{txx}+4u^2u_x=3uu_xu_{xx}+u^2u_{xxx}, \quad t>0, x\in \R.
}

(4) Theorem \ref{thm1} also holds for the two-component CH, DP, b-family equations. Namely, we have norm inflation in $B^{1+1/p}_{p,r}\times B^{1/p}_{p,r}$ for $1<r\leq \infty$, $1\leq p\leq \infty$.
\end{rem}

\section{Preliminaries}

\subsection{Notations}
Throughout this paper, we use $C$ to denote universal contant which may vary from line to line. For $1\leq r\leq \infty$, we denote its conjugate number $r'=\frac{r}{r-1}$. For function $f$ defined on $\R$ or $\T$, we define its Fourier transform denoted by $\hat{f}(\xi)$ as
\EQ{
\hat{f}(\xi)=\int_\R f(x)e^{-2\pi ix\xi}dx, \quad \xi\in \R; \quad
 \hat{f}(\xi)=\frac{1}{2\pi}\int_\T f(x)e^{-i\xi x}dx, \quad \xi \in  \Z.
}
For convenience, we use $\E^*$ to denote $\R$ or $\Z$, endowed with their natural measure $\mu$. $H^s(\E)$ denotes the usual Sobolev space  consisting of tempered distributions $f\in \mathcal{S'}(\E)$ such that $\hat{f}\in L^2_{loc}(\E)$ and
\[\|u\|_{H^s(\E)}:=\big[\int_{\E^*}(1+|\xi|^2)^s|\hat{f}(\xi)|^2d\mu(\xi)\big]^{1/2}<\infty.\]

Let $\eta: \R\to [0, 1]$ be an even, smooth, non-negative and radially decreasing
function which is supported in $\{\xi:|\xi|\leq \frac{8}{5}\}$ and
$\eta\equiv 1$ for $|\xi|\leq \frac{5}{4}$. For $k\in \Z$, let
$\chi_k(\xi)=\eta(\frac{\xi}{2^k})-\eta(\frac{\xi}{2^{k-1}})$ and $\chi_{\leq
k}(\xi)=\eta(\frac{\xi}{2^k})$, and define Littlewood-Paley operators $P_k, P_{\leq k}$ on $L^2(\E)$ by
$\widehat{P_ku}(\xi)=\chi_k(|\xi|)\widehat{f}(\xi),\,\widehat{P_{\leq
k}u}(\xi)=\chi_{\leq k}(|\xi|)\widehat{f}(\xi)$. We can then define the Besov space $B^s_{p,r}$ with the norm $\norm{f}_{B^s_{p,r}}=\normb{2^{ks}\norm{P_k f}_{L^p(\E)}}_{l_k^r}$.

\subsection{Useful Lemmas}
In this subsection, we collect some results that we need.
The first one is the well-posedness result (see \cite{Li-Olver, Rodriguez-Blanco}).  For other equations in Remark \ref{remmain}, we have the same results.

\begin{lem}[\cite{Li-Olver, Rodriguez-Blanco}]\label{lem1}
Given $u_0(x)\in H^s$, $s>3/2$, there exists a maximal $T\ge \tilde{T}(\norm{u_0}_{H^{\frac{3}{2}+}})>0$ and a unique solution $u$
to \eqref{eq:CH} such that
\[u\in C([0,T);H^s)\cap C^1([0,T);H^{s-1}).\]
Moreover, the solution depends continuously on the initial data, i.e. for any $T'<T $
the mapping $u_0\mapsto u: H^s\mapsto  C([0,T'];H^s)\cap C^1([0,T'];H^{s-1})$ is continuous  on a $ H^s $-neighborhood
 of $ u_0 $, and if $T<\infty$, then $\lim_{t\to T^-}\norm{u(t)}_{H^s}=\infty$.
\end{lem}

The proof of the well-posedness relies on rewriting the equation into a perturbation of Burgers equation. For example, the general $b$-family equations are equivalent to
\begin{align}\label{eq:CH1}
u_t+uu_x=-\partial_x\Lambda^{-2}(\frac{b}{2}u^2+\frac{3-b}{2}u_x^2)=-\partial_xp*(\frac{b}{2}u^2+\frac{3-b}{2}u_x^2);
\end{align}
where $\Lambda=(1-\partial_x^2)^{1/2}$ and
\EQ{
p(x)=\CAS{e^{-|x|}/2, &\quad x\in \R; \\\dfrac{\cosh(x-[x]-1/2)}{2\sinh(1/2)}, &\quad x\in \T.}
}
Thus, semilinear well-posedness are not expected, indeed, non uniform continuity of the flow-map for CH in $H^s$ for large $s$ was shown in \cite{Himonas-K}. The CH equation is integrable and has infinite many conservation laws, in particular, the energy conservation
\EQ{
E(u)=\int u^2+u_x^2dx.
}
So $H^1$ is the most natural setting for the physically relevant weak solutions. Global weak solution in $H^1$ was eastablished in \cite{Xin-Zhang}.  In the $H^1$ setting one has to associate a measure to keep track of possible singularities and of whether one imposes conservation of energy or permits dissipation (see \cite{C-Molinet, B-C1,B-C2}).
On the other hand, one can use the classical energy methods to \eqref{eq:CH1} and prove well-posedness in $H^s$ for $s>3/2$ in the sense of Hadamard.  The restriction $s>3/2$ comes from the energy estimates: let $u$ be a smooth solution to \eqref{eq:CH1}, then
\EQn{\label{eq:energy estimates}
\frac{d}{dt}\norm{u}_{H^s}^2\leq C\norm{u_x}_{L^\infty} \norm{u}_{H^s}^2.
}
We need $s>3/2$ to ensure the embedding $H^s(\R)\hookrightarrow Lip$ with $s>3/2$ ($Lip$ here denotes the bounded Lipschitz functions). For $s=3/2$, we could replace $H^s$ by $B^{3/2}_{2,1}$. The energy estimates \eqref{eq:energy estimates} hold for all the equations in Remark \ref{remmain} except Novikov equation \eqref{eq:Novikov}.
For $s=2$, we can derive it by just differentiating the equation and integrating by part.  For $3/2<s<2$, we need the commutator estimates.  For Novikov equation, we have the energy estimates:  let $u$ be a smooth solution to \eqref{eq:Novikov}, then
\EQn{\label{eq: Nov energy}
\frac{d}{dt}\norm{u}_{H^2}^2\leq C\norm{u}_{L^\infty}\norm{u_x}_{L^\infty} \norm{u}_{H^2}^2.
}
Indeed, let $y=u-\partial_x^2u$.  Then the Novikov equation is equivalent to
\begin{align*}
y_t+u^2y_x+3yu_xu=0.
\end{align*}
Multiplying $y$ on both sides of the above equation and then integrating by parts we get \eqref{eq: Nov energy}.

A crucial difference between CH and KdV is that for CH equation smooth data can develop singularity in finite time. We have the following classical blow-up result (see \cite{C-E1}). For other equations in Remark \ref{remmain} except Novikov equation we have the similar blow-up results. The blow-up of smooth solutions for CH
can only occur in the form of wave breaking: the solution remains bounded but its slope becomes unbounded (see \cite{C-E3, Con, C-E4}).

\begin{lem}[\cite{C-E1}]\label{lem3}
Assume  $u_0(x)\in H^3(\E)$ is a real-valued odd function and $u'_0(0)<0$.
Then the corresponding strong solution to \eqref{eq:CH} blows up in finite time. Moreover,
the maximal time of existence is bounded by $2/|u'_0(0)|$.
\end{lem}

The proof of Lemma \ref{lem3} is quite simple based on the observation of preservation of anti-symmetry $u(t,x)\to -u(t,-x)$ under the flow of the CH equation \eqref{eq:CH}. We sketch it here.  For $u_0(x)\in H^3$ odd, then the solution of CH equation satisfies $u(t,x)=-u(t,-x)$ and thus $u_{xx}(t,0)=0$.  Setting $g(t):=u_x(t,0)$ for $t\in [0,T)$, we deduce from \eqref{eq:CH1} that
\begin{align*}
\dfrac{d}{dt}g(t)+\dfrac{1}{2}g^2(t)=-\left(p*(u^2+\dfrac{u_x^2}{2})\right) (0)\leq0.
\end{align*}
Thus, we have
\begin{align*}
\dfrac{d}{dt}g(t)\leq -\dfrac{1}{2}g^2(t), \quad \ t\in[0,T).
\end{align*}
Since $g(0)<0$, then $g(t)<0$ and
\begin{align*}
0>\dfrac{1}{g(t)}\geq\dfrac{1}{g_0}+\dfrac{t}{2}, \quad \ t\in[0,T),
\end{align*}
which implies that $T<-2/u'_0(0)$.

For general $b$-family equations, the above arguments also work for $1<b\leq 3$. Indeed, similarly we get
\begin{align*}
\dfrac{d}{dt}g(t)+\dfrac{b-1}{2}g^2(t)=-\left(p*(\frac{b}{2}u^2+\frac{3-b}{2}u_x^2)\right) (0)\leq0.
\end{align*}
Then we get the maximal time of existence $T<-\frac{2}{(b-1)u'_0(0)}$.
For the Novikov equation on the real-line, the above argument does not work but there are similar blow-up results (e.g. see \cite{YanLZ}). However, no explicit initial data was constructed in \cite{YanLZ}. To apply their result, we construct an explicit example (see Remark \ref{remlast}). We are not aware of any blow-up results for the periodic Novikov equation.
\begin{lem}[\cite{YanLZ}]\label{lem4}
Assume  $u_0(x)\in H^3(\R)$ is a real-valued function and let $y_0=(1-\partial_x^2)u_0$. If $y_0(0)=0$, $y_0(x)\geq 0$ for $x\leq 0$ and $y_0(x)\leq 0$ for $x\geq 0$, and
\EQ{
u_0(0)u_0'(0)<-\frac{1}{2}\norm{u_0}_{H^1}^2.
}
Then the corresponding strong solution to \eqref{eq:Novikov} blows up in finite time. Moreover,
the maximal time of existence $T$ is bounded as follows
\[T\leq \min\left\{-\frac{2}{(1-\delta)m(0)}, \frac{2}{\norm{u_0}_{H^1}^2}\ln\frac{m(0)-\frac{1}{2}\norm{u_0}_{H^1}^2}{m(0)+\frac{1}{2}\norm{u_0}_{H^1}^2}\right\}\]
where $-\sqrt{\delta}m(0)=\frac{1}{2}\norm{u_0}_{H^1}^2$, $m(0)=u_0(0)u_0'(0)$.
\end{lem}

\section{Proof of Theorem \ref{thm1}}

In this section we prove Theorem \ref{thm1}.
We rely on the following Gronwall type estimates.

\begin{lem}\label{lem:gronwall}
Let $I=[0,T)$, $T>0$ could be infinity. Assume $A(t)\in C^1(I)$, $A(t)>0$ and there exists a constant $B>0$ such that
\EQn{\label{eq:At}
\dfrac{d}{dt}A(t)\leq BA(t)\ln (2+A(t)), \quad \forall\ t\in I.
}
Then we have
\EQn{
A(t)\leq (2+A(0))^{e^{Bt}}, \quad \forall\ t\in I.
}
\end{lem}
\begin{proof}
By assumption we have
\EQ{
\frac{A'(s)}{[2+A(s)]\ln [2+A(s)]}\leq B, \quad \forall\ s\in I.
}
Integrating in $s$ over the interval $[0,t)$ for $t\in I$, we obtain the bound.
\end{proof}

\begin{lem}\label{leminfty}
Assume $u\in H^2(\E)$.  We have
\EQ{
\norm{u_x}_{L^\infty(\E)}\leq C \norm{u}_{B^1_{\infty,\infty}}\cdot\log_2(2+\norm{u}_{H^2}^2)+C.
}
\end{lem}
\begin{proof}
Fixing an integer $N>0$, we get
\EQ{
\norm{u_x}_{L^\infty}\leq& \sum_{k\leq N-1}\norm{P_ku_x}_{L^\infty}+\sum_{k\geq N}\norm{P_ku_x}_{L^\infty}\\
\leq& CN \norm{u}_{B^1_{\infty,\infty}}+C\sum_{k\geq N}2^{-k/2}2^{2k}\norm{P_ku}_{L^2}\\
\leq& CN \norm{u}_{B^1_{\infty,\infty}}+C2^{-N/2}\norm{u}_{H^2}.
}
Setting $N=\log_2(2+\norm{u}_{H^2}^2)$, we complete the proof.
\end{proof}

Now we are ready to prove Theorem \ref{thm1}.
Fix $1\leq p\leq \infty, 1<r\leq \infty$ and $\e>0$. We define $h(x)$
\EQ{
h(x)=\sum_{k\geq 1} \frac{1}{2^{2k}k^{\frac{2}{1+r}}}h_k(x)
}
with $h_k$ given by the Fourier transform $\widehat{h}_k(\xi)=i2^{-k}\xi\tilde{\chi}(2^{-k}\xi)$, $\xi\in \E^*$, where $\tilde{\chi}$ is an even non-negative, non-zero $C_0^\infty$ function such that $\tilde{\chi}\chi_0=\tilde{\chi}$.
Then clearly we see that $h$ is real-valued odd function and $(P_k h)(x)=\frac{1}{2^{2k}k^{\frac{2}{1+r}}}h_k(x)$. We also have $\norm{P_kh}_{L^p}\sim \frac{2^{k/p'}}{2^{2k}k^{\frac{2}{1+r}}}$ and   thus
\EQ{
\norm{h}_{B^{1+1/p}_{p,q}(\E)}\sim \normo{\frac{1}{k^{\frac{2}{1+r}}}}_{l_k^q}.
}
From this we see  that $h\in B^{1+1/p}_{p,r}(\E)\setminus B^{1+1/p}_{p,1}(\E)$, and
\[h'(0)=\int \widehat{h'}(\xi)d\xi=\int i2\pi\xi \widehat{h}(\xi)d\xi=-\infty.\]
For $ \e>0 $, we take $u_{0,\e}=\norm{h}_{B^{1+1/p}_{p,r}}^{-1}\cdot \e P_{\leq K}(h)$ with $K$ sufficiently large such that $u_{0,\e} '(0)<-2\e^{-10}$. Then $u_{0,\e}$ is a real valued odd function, $u_{0,\e}\in H^\infty(\E)$, $\norm{u_{0,\e}}_{B^{1+1/p}_{p,r}}\leq \e$. By Lemma \ref{lem1}
and Lemma \ref{lem3}
, there is a unique associated solution $u_\e \in C([0,T); H^\infty(\E))$ with a maximal lifespan $T_\e<\e^{10}$.
To prove Theorem \ref{thm1} it suffices to show
\EQn{\label{Claim2}\limsup_{t\to T_\e^-}\norm{u_\e (t)}_{B^{1}_{\infty,\infty}}=\infty.}

We prove \eqref{Claim2} by contradiction. If \eqref{Claim2} fails, then $\exists M_\e >1$ such that \[\sup_{t\in [0,T_\e)}\norm{u_\e (t)}_{B^{1}_{\infty,\infty}}\leq M_\e.\]
By the energy estimates \eqref{eq:energy estimates} and Lemma \ref{leminfty}, we get
\EQ{
\frac{d}{dt}\norm{u_\e}_{H^2}^2\leq& C\norm{u_{\e,x}}_{L^\infty} \norm{u_\e}_{H^2}^2\\
\leq& C( \norm{u_\e}_{B^1_{\infty,\infty}}\log_2(2+\norm{u_\e}_{H^2}^2)+1)\norm{u_\e}_{H^2}^2\\
\leq& CM_\e \norm{u_\e}_{H^2}^2\log_2(2+\norm{u_\e}_{H^2}^2).
}
Using Gronwall inequality in Lemma \ref{lem:gronwall} we get $\sup_{t\in [0,T_\e)}\norm{u_\e(t)}^2_{H^2}<\infty$ which contradicts to the blow-up criteria in Lemma \ref{lem1}.

\begin{rem}\label{remlast}
To show the norm inflation for the Novikov equation on the real-line, one  uses the energy estimate \eqref{eq: Nov energy} and the blow-up result in Lemma \ref{lem4} as  in the above proof.  The key point is to construct such an initial data.  Fixing $p\in [1,\infty]$ and $r\in (1,\infty]$, we define
\EQ{
u_0(x)=\sum_{k= 1}^K \frac{1}{k^{\frac{2}{1+r}}}\Lambda^{-2} (\phi_k)( x)
}
where $K$ is a large integer determined later and $\phi_k=2^k\phi(2^k x)$ with $\phi\in C_0^\infty (\R)$ given by $\phi(x)=\eta(x+2)-\eta(x-200)$.

Obviously $u_0$ is a real-valued $H^\infty$ function and we can verify the following properties.

1.  For $y_0(x)=\Lambda^2 u_0(x)=\sum_{k= 1}^K \frac{1}{k^{\frac{2}{1+r}}}\phi_k( x)$, it is easy to see that $y_0(x)\geq 0$ for $x\leq 0$; $y_0(x)\leq 0$ for $x\geq 0$ and $y_0(0)=0$.

2. $u_0\in B^{1+1/p}_{p, r}$ and $\norm{u_0}_{B^{1+1/p}_{p, r}}\leq C_r$ uniformly in $K$.  Indeed,
since
\[\widehat{\phi}(\xi)=\hat{\eta}(\xi)(e^{4\pi i\xi}-e^{-400 \pi i\xi}),\]
we have $\widehat{\phi}(0)=0$.  So $\widehat{\phi}(\xi)$ is a Schwartz function localized at $|\xi| \sim 1$.  Rigorously,
\EQ{
P_k [\Lambda^{-2}\phi_j](x)=\int (1+4\pi^2\xi^2)^{-1}\hat{\phi}(2^{-j}\xi)e^{ix\xi}\chi (2^{-k}\xi)d\xi.
}
Then we have $\norm{P_k [\Lambda^{-2}\phi_j]}_{L^p}\leq C 2^{-2k}2^{k/p'}2^{-|j-k|}$ and thus $\norm{u_0}_{B^{1+1/p}_{p, r}}\leq C_r$.

3. $u_0(0)>C$ uniformly in $K$ and $u_0'(0)\to -\infty$ as $K\to \infty$.  Indeed,
\EQ{
u_0(0)=&\sum_{k= 1}^K \frac{1}{2k^{\frac{2}{1+r}}}\int e^{-2^{-k}|x|}\phi( x)dx\\
=&\sum_{k= 1}^K \frac{1}{2 k^{\frac{2}{1+r}}}\int (e^{-2^{-k}|x-2|}-e^{-2^{-k}|x+200|})\eta( x)dx\\
\geq&\frac12\int (e^{-2^{-1}|x-2|}-e^{-2^{-1}|x+200|})\eta( x)dx\geq C.
}
Moreover,
\EQ{
u_0'(0)=&\int 2\pi i\xi \widehat{u_0}(\xi)d\xi=\sum_{k= 1}^K \frac{1}{k^{\frac{2}{1+r}}}\int 2\pi i\xi  (1+4\pi^2\xi^2)^{-1}\hat{\phi}(2^{-k}\xi)d\xi\\
=&\sum_{k= 1}^K \frac{1}{k^{\frac{2}{1+r}}}\int \frac{i\hat{\phi}(2^{-k}\xi)}{2\pi\xi} d\xi-\sum_{k= 1}^K \frac{1}{k^{\frac{2}{1+r}}}\int  \frac{i\hat{\phi}(2^{-k}\xi)}{2\pi\xi (1+4\pi^2\xi^2)}d\xi\\
:=&I+II.
}
For the term $II$, using the fact $|\hat{\phi}(2^{-k}\xi)|\leq C 2^{-k}|\xi|$ we get $|II|\leq C$.  On the other hand,
\EQ{
I=&\sum_{k= 1}^K \frac{1}{k^{\frac{2}{1+r}}}\int \frac{i(e^{4\pi i\xi}-e^{-400 \pi i\xi})\hat{\eta}(\xi)}{2\pi\xi} d\xi\\
=&\sum_{k= 1}^K \frac{-1}{k^{\frac{2}{1+r}}}\int \frac{\sin (4\pi \xi)+\sin(400 \pi \xi)}{2\pi\xi}\hat{\eta}(\xi) d\xi.
}
Using the fact $\widehat{1_{[-A,A]}(x)}(\xi)=\frac{\sin (2\pi A\xi)}{\pi\xi}$ and Parseval equality we get
\EQ{
I=&\sum_{k= 1}^K \frac{-1}{2k^{\frac{2}{1+r}}}\brk{\int_{-2}^2\eta(x) dx+\int_{-200}^{200}\eta(x) dx}\to -\infty, \quad K\to \infty.
}

With the above properties we see that taking intial data $ u_{0,\e}=C^{-1}\e u_0$, then $\norm{u_{0,\e}}_{B^{1+1/p}_{p,r}}\leq \e$ uniformly in $K$ and the maximal time of existence of the solution tends to $0$ (as $K\to \infty$) by Lemma \ref{lem4}. Using the energy estimate \eqref{eq: Nov energy} and blowup criteria we can prove Theorem \ref{thm1} for Novikov equation on the real-line.
\end{rem}

\,

\noindent\textbf{Acknowledgments} \ Z. Guo was  partially supported by ARC DP170101060.
X. Liu was partially supported by the Fundamental Research Funds for the Central
Universities (No.2018QNA34 and No.2017XKZD11).  Z. Yin was partially supported by NSFC (No.11671407 and No.11271382), FDCT (No.098/2013/A3), Guangdong Special Support Program (No.8-2015), and
the key project of NSF of Guangdong Province (No.2016A030311004).

\end{document}